\documentclass[12pt]{amsart}
\usepackage[top=30mm, bottom=30mm, left=30mm, right=30mm]{geometry}
\usepackage{hyperref}
\usepackage{graphicx}
\usepackage{amssymb}
\usepackage{amsmath}
\usepackage{amsthm}
\usepackage{xcolor}

\newtheorem{thm}{Theorem}[section]
\newtheorem{cor}[thm]{Corollary}
\newtheorem{lem}[thm]{Lemma}
\newtheorem{exmp}[thm]{Example}
\newtheorem{prop}[thm]{Proposition}
\newtheorem{rem}[thm]{Remark}
\theoremstyle{definition}
\newtheorem{defn}[thm]{Definition}
\newtheorem{que}{Question}
\numberwithin{equation}{section}

\newcommand{\N}{\mathbb{N}_0}

\newcommand{\R}{\mathbb{R}}

\newcommand{\set}[1]{\left\{#1\right\}}
\newcommand{\eps}{\varepsilon}

\DeclareMathOperator{\supp}{supp}
\DeclareMathOperator{\diam}{diam}

\begin{document}
\title{Shadowing property, weak mixing and regular recurrence}
\author[J. Li]{Jian Li}
\date{\today}
\address[J. Li]{Department of Mathematics, Shantou University, Shantou, Guangdong, 515063, P.R. China}
\email{lijian09@mail.ustc.edu.cn}

\author[P. Oprocha]{Piotr Oprocha}
\address[P.~Oprocha]{AGH University of Science and Technology, Faculty of Applied
Mathematics, al.
Mickiewicza 30, 30-059 Krak\'ow, Poland -- \and --  Institute of Mathematics\\ Polish Academy of Sciences\\
ul. \'Sniadeckich 8, 00-956 Warszawa, Poland} \email{oprocha@agh.edu.pl}

\date{\today}

\begin{abstract}
We show that a non-wandering dynamical system with the shadowing property is either equicontinuous or has positive entropy
and that in this context uniformly positive entropy is equivalent to weak mixing.
We also show that weak mixing together with the  shadowing property imply the specification property
with a special kind of regularity in tracing (a weaker version of periodic specification property).
This in turn
implies that the set of ergodic measures supported on the closures of orbits of regularly recurrent points %Toeplitz systems
is dense in the space of all invariant measures
(in particular, invariant measures in such a system form the Poulsen simplex, up to an affine homeomorphism).
\end{abstract}
\maketitle

\section{Introduction}
The notion of shadowing property (also known under the name of pseudo-orbit tracing property)
was introduced in the fundamental works of Anosov and Bowen on hyperbolic aspect of differentiable dynamics.
It was later discovered that this notion can lead to many interesting results
in the study of measure theoretic and topological properties of dynamical systems on compact metric spaces.
These studies are now among classical results of ergodic theory and qualitative theory of
dynamical systems (see~\cite{AH94} and~\cite{Denker}).

In~\cite{M11} Moothathu proved that in every
non-wandering dynamical system with the shadowing property  the set of minimal points is dense and
recently this result was extended in \cite{MO},
by showing that the set of regularly recurrent points is also dense.
If a non-wandering dynamical system with the shadowing property is sensitive,
various types of minimal subsystems are present in the dynamics,
such as sensitive almost 1-1 extension of odometers and minimal subsystems with positive entropy.
In fact, every non-empty open subset contains an extension of the full shift for some power of the action.
It is shown in~\cite{M11} that for a dynamical system with the shadowing property,
a sensitive point from the non-wandering set is an entropy point
and we will show that there is a full shift factor for some power of the action
in every neighborhood of a sensitive point.
As an immediate consequence we obtain a kind of dichotomy:
a non-wandering dynamical system with the shadowing property
is either equicontinuous or has positive entropy.
We also obtain a few equivalent conditions for positive entropy in systems with the shadowing property.
It is known (see~\cite{L11}) that for a continuous map $f\colon [0,1]\to [0,1]$
weak mixing is equivalent to uniformly positive entropy of all orders.
In the present paper we show that the same equivalence is true for non-trivial systems with the shadowing property.

It was first proved by Bowen that if a weakly mixing system  has the shadowing property,
then it satisfies the specification property and if additionally the system is positively expansive,
then it satisfies the periodic specification property (e.g. see~\cite[Proposition~23.20]{Denker}).
We will combine this result with the above mentioned technique from~\cite{MO} to show
that if a dynamical system with the shadowing property is weakly mixing then it has
the specification property with some regularity in tracing.

It is shown in~\cite{S70} that
if a dynamical system satisfies the periodic specification property,
then the set of atomic measures which are uniformly distributed on the periodic orbits
is dense in the set of invariant measures.
This classical result was later extended by various authors in many different settings,
including maps on the unit interval~\cite{FH} where there is no chance for any kind of expansiveness.
It is also known that positively expansive systems with the shadowing property have dense sets of periodic points.
But there are also weakly mixing systems with the shadowing property and without periodic points
(e.g. see Example~\ref{exmp:shadowing-non-period}),
therefore they are not positively expansive.
While there is no possibility to obtain the same result
(e.g. on the structure of the sets of invariant measures) in general setting,
some part of qualitative behavior survives if we drop the positive expansiveness assumption.
To be specific, we show that if a weakly mixing system has the shadowing property,
then the set of ergodic invariant measures supported on the closures of orbits of regularly recurrent points
is dense in the set of invariant measures.
Since extreme points of invariant measure are exactly the ergodic measures,
invariant measures in such a system form the Poulsen simplex, up to an affine homeomorphism.
For more details about the Poulsen simplex and its connections with dynamical systems,
we refer the reader to~\cite{GW97} and~\cite{LOS78}.

\section{Preliminaries}

Throughout this paper, let $\mathbb{N}$, $\N$ and $\mathbb{R}$ denote the set of all positive integers,
non-negative integers and real numbers, respectively.
Let $(X,d)$ be a metric space.
\emph{Open} and \emph{closed balls} of radius $r >0$ centered at a point $x$ in $X$
are denoted by $B(x,r)$ and $\overline{B}(x,r)$, respectively.
A subset $A$ of $X$ is \emph{residual} if it is a dense $G_\delta$ set.

\subsection{Topological dynamics}
A \emph{dynamical system} is a pair $(X,f)$,
where $X$ is a compact metric space with a metric $d$ and $f\colon X\to X$ is a continuous map.

The \emph{orbit} of a point $x\in X$ is the set $Orb(x,f)=\{f^n(x):\ n\in\N\}$.
The set of limit points of the orbit $Orb(x,f)$ is called the \emph{$\omega$-limit set}
of $x$, and is denoted by $\omega(x,f)$.

A subset $D$ of $X$ is \emph{$f$-invariant} (or simply \emph{invariant}) if $f(D)\subset D$.
A non-empty closed invariant subset $D$ of $X$ is \emph{minimal}, if $\overline{Orb(x,f)} = D$ for every $x \in D$.
A point $x \in  X$ is \emph{minimal} if it is contained in some minimal subset of $X$.

A point $x \in X$ is \emph{periodic with least period $n$} if $n$ is the smallest positive integer satisfying $f^n(x) = x$;
\emph{recurrent} if for every neighborhood $U$ of $x$ there exists $k\in\mathbb{N}$ such that $f^k(x)\in U$;
\emph{regularly recurrent} if for every open neighborhood $U$ of $x$,
there exists $k\in\mathbb{N}$ such that $f^{kn}(x)\in U$ for all $n\in\N$.
It is well known that a point $x\in X$ is minimal if and only if
for every open neighborhood $U$ of $x$, there exists $N> 0$ such that $f^k(x)\in U$
for some $k\in [n,n+N]$ and every $n\in\N$.
Note that every periodic point is regularly recurrent, every regularly recurrent point is minimal
and every minimal point is recurrent (but not vice-versa).
Denote by $P(f)$, $R(f)$, $M(f)$ and $RR(f)$, respectively,
the set of all periodic, recurrent, minimal and regularly recurrent points of $f$.
We say that $x \in X$ is a \emph{non-wandering point} if  for every neighborhood $U$ of $x$,
there exists $k\in\mathbb{N}$ such that $f^k(U)\cap U\neq\emptyset$.
The set of all non-wandering points of $f$ is denoted as $\Omega(f)$.
Observe that $\Omega(f)$ is closed and $f$-invariant.
If $\Omega(f) = X$, the system is said to be \emph{non-wandering}.

We say that $f$ is \emph{transitive} if for every pair of non-empty open subsets $U$ and $V$
of $X$ there is $k\in\mathbb{N}$ such that $f^k(U)\cap  V \neq\emptyset$;
\emph{weakly mixing} if $f \times f$ is transitive;
\emph{strongly mixing}  if for every pair of non-empty open sets $U$ and $V$ of $X$ there is an $N > 0$
such that $f^n(U)\cap  V \neq\emptyset$ for all $n \geq N$.
We say that $x\in X$ is a \emph{transitive point} if $\omega(x,f)=X$.
It is well known that a dynamical system $(X,f)$ is transitive if and only if
the set of transitive points is residual in $X$.

A dynamical system $(X,f)$ is \emph{equicontinuous} if for every $\eps>0$, there is $\delta>0$
with the property that $d(x,y)<\delta$ implies $d(f^n(x,),f^n(y))<\eps$
for every $n\in\mathbb{N}$ and every $x,y\in X$.
A point $x\in X$ is \emph{equicontinuous} if for every $\eps>0$, there is $\delta>0$
with the property that $d(x,y)<\delta$ implies $d(f^n(x,),f^n(y))<\eps$
for every $n\in\mathbb{N}$ and every $y\in X$.
By the compactness of $X$, a system $(X,f)$ is equicontinuous if and only if
every point in $X$ is equicontinuous.

A dynamical system $(X,f)$ is \emph{sensitive} if there exists $\delta>0$ such that
for any non-empty open subset $U$ of $X$, we have $\diam (f^n(U)) > \delta$ for some $n \in \mathbb{N}$.
A point $x\in X$ is \emph{sensitive} if it is not equicontinuous, that is
there exists $\delta>0$ such that for any neighborhood $U$ of $x$ there is $n>0$ such that
$\diam ( f^n(U)) > \delta$. Clearly, if $(X,f)$ is sensitive then every point $x\in X$ is sensitive.

Let $(X,f)$ and $(Y,g)$ be two dynamical systems. If there is a continuous surjection
$\pi: X \to Y$ with $\pi\circ f = g\circ \pi$,
then we say that $\pi$ is a \emph{factor map}, the system
$(Y,g)$ is a \emph{factor} of $(X,f)$ or $(X, f)$ is an \emph{extension} of $(Y,g)$.
We say that $(X,f)$ is almost 1-1 extension of $(Y,g)$ if $Z=\{x\in X: \pi^{-1}(\pi(x))=\{x\}\}$
is residual in $X$. If $\pi$ is a homeomorphism, then we say that $\pi$ is a \emph{conjugacy} and
dynamical systems $(X,f)$ and $(Y,g)$ are \emph{conjugate}. Conjugate dynamical systems can
be considered the same from the dynamical point of view.

\subsection{Shifts and Odometers}
For any integer $d\geq 1$, the space $\{0,1,\dotsc,d\}^{\N}$ is a Cantor space with respect to the product topology.
We write elements of $\{0,1,\dotsc,d\}^{\N}$ as $\alpha=a_0a_1a_2\dotsc$.
The shift map $\sigma:\{0,1,\dotsc,d\}^{\N}\to\{0,1,\dotsc,d\}^{\N}$ is defined by
the condition that $\sigma(\alpha)_n=\alpha_{n+1}$ for $n\in\N$.
It is not hard to check that $\sigma$ is a continuous surjection. The dynamical system
$(\{0,1,\dotsc,d\}^{\N},\sigma)$ is called the \emph{full shift}.

Let $s=(s_j)_{j=1}^\infty$ be a sequence of positive integers
such that $s_j$ divides $s_{j+1}$.
Let $X(j)=\{0,1,\dotsc,s_j-1\}$ and
$X_s=\{x\in\prod_{j=1}^\infty X(j): x_{j+1}\equiv x_j \pmod{s_j}\}$.
Let $f:X_s\to X_s$, $x\mapsto y$, where $y_j=x_j+1 \pmod{s_j}$ for each $j=1,2,\dotsc$.
The dynamical system $(X_s,f)$ is called an \emph{odometer} (or an \emph{adding machine})
defined by the sequence $s=(s_j)_{j=1}^\infty$.
Some authors require that the sequence $(s_j)_{j=1}^\infty$ starts with at least $2$ and is strictly increasing,
but lack of such restrictions has no formal consequences other than admitting (as odometers) periodic orbits
including the trivial one.
It is clear that $(X_s,f)$ is minimal and equicontinuous, and every point in $X_s$ is regularly recurrent.
In the further parts of the paper we will refer to the fact (e.g. see~\cite[Theorem 5.1]{D05}) that the orbit closure $\overline{Orb(x,f)}$
of any regularly recurrent point $x\in X$ is an almost 1-1 extension of an odometer.
See the survey~\cite{D05} by Downarowicz for more details on odometers.

\subsection{The shadowing property}
Let $(X,f)$ be a dynamical system. Fix any $\eps>0$
and $\delta>0$. A sequence $\{x_n\}_{n=0}^\infty$ in $X$ is a \emph{$\delta$-pseudo orbit}
if $d(f(x_n),x_{n+1})<\delta$ for all $n=0,1,2,\dotsc$.
A point $x\in X$ is \emph{$\eps$-tracing} a pseudo-orbit $\{x_n\}_{n=0}^\infty$
if $d(f^n(x),x_n)<\eps$ for all $n=0,1,\dotsc$.
We say that the system $(X,f)$ (or the map $f$) has the \emph{shadowing property} (or \emph{pseudo-orbit tracing property})
if for every $\eps>0$ there is $\delta>0$ such that
every $\delta$-pseudo orbit of $f$ is $\eps$-traced by some point in $X$.

If $x,y\in X$ then an $\eps$-chain of length $n>1$ (a finite $\eps$-pseudo orbit)
form $x$ to $y$ is any sequence $x_1,\ldots, x_n$
such that $x_1=x$, $x_n=y$ and $d(f(x_i),x_{i+1})<\eps$ for $i=1,2,\ldots, n-1$.
We say that $(X,f)$ is
\emph{chain mixing} if for every $\eps>0$
there is $M>0$ such that for any two points $x,y\in X$ and any $n\geq M$ there is an $\eps$-chain
of length $n$ from $x$ to $y$.

The following facts highlight strong connections between the non-wandering set and the shadowing property.
\begin{thm}[\cite{AH94,M11}]\label{thm:shad:omega}
If $f$ has the shadowing property, then $f|_{\Omega(f)}$ also has the shadowing property.
\end{thm}

Additionally, in dynamical systems with the shadowing property there are many minimal points.

\begin{thm}[\cite{M11}]\label{thm:nw-dmp}
Let $(X,f)$ be a non-wandering dynamical system with the shadowing property.
 Then the set $M(f)$ of minimal points is dense in $X$.
\end{thm}

\begin{thm}[\cite{MO}]
 Let $(X,f)$ be a non-wandering dynamical system with the shadowing property.
 Then the set $RR(f)$ of regularly recurrent points is dense in $X$.
\end{thm}

The following fact can be easily deduced from known results.
We present a proof here for completeness.

\begin{prop} \label{prop:eq-RR}
Let $(X,f)$ be a non-wandering system. If $(X,f)$ is equicontinuous then $M(f)=X$
(i.e. $X$ is a union of minimal equicontinuous systems).
If additionally $(X,f)$ has the shadowing property then $RR(f)=X$
(and each minimal subsystem is conjugated to an odometer).
\end{prop}
\begin{proof}
Fix a point $x\in X$. Then it is not hard to see that $x$ is a recurrent point (see \cite{AAB93}).
This implies (e.g. again by \cite{AAB93}) that $(\overline{Orb(x,f)},f)$ is an equicontinuous minimal subsystem of $(X,f)$.
By \cite{M11} any equicontinuous surjective dynamical $(X,f)$ has the shadowing property if and only if $X$ is totally
disconnected. Then $(\overline{Orb(x,f)},f)$ is either a finite set (a periodic orbit)
or a Cantor set, and then it is conjugate to an odometer (see \cite{Akin2} or \cite{MY02}).
In each case $x$ is regularly recurrent.
\end{proof}

\section{Dichotomy results for dynamical systems with the shadowing property}

It is shown in~\cite{M11} that for a dynamical system with the shadowing property,
a sensitive point from the non-wandering set is an entropy point (see~\cite{YZ07} for the definition).
In fact, we can show that there is a full shift factor in every neighborhood of a sensitive point.
\begin{prop}\label{prop:sensitive-point}
Let $(X,f)$ be a dynamical system with the shadowing property.
If $u\in X$ is a sensitive point of $(\Omega(f),f)$,
then for every neighborhood $U$ of $u$, there exists a positive integer $m$,
a subsystem $(Y,f^m)$ of $(X,f^m)$ with $Y\subset U$
and a factor map $\pi:(Y,f^m)\to (\{0,1\}^{\N},\sigma)$.
\end{prop}

\begin{proof}
Since $u$ is a sensitive point, there is a positive number $\lambda$ with the property that
for any neighborhood $V$ of $u$, we have $\diam ( f^n(V\cap \Omega(f))) > \lambda$ for some $n \in \mathbb{N}$.
Pick a positive number $\eps<\lambda/8$ such that $B(u,3\eps)\subset U$.
By Theorem~\ref{thm:shad:omega} $(\Omega(f),f)$ also has the shadowing property.
Let $\delta$ be a constant provided for $\eps>0$ by the shadowing property of $f|_{\Omega(f)}$.
By the definition of $\lambda$,
there exists a positive integer $r$ and two points $v_1,v_2\in \Omega(f)\cap B(u,\delta/4)$
such that $d(f^r(v_1),f^r(v_2))>\lambda$.
By Theorem~\ref{thm:nw-dmp} $(\Omega(f),f)$ has a dense set of minimal points,
the same is true for $(\Omega(f)\times \Omega(f),f\times f)$.
By the continuity of $f^r$, there exists a minimal point $(w_1,w_2)$ of $(\Omega(f)\times \Omega(f),f\times f)$
such that $d(v_1,w_1)<\delta/4$,  $d(v_2,w_2)<\delta/4$ and
$d(f^r(v_1),f^r(w_1))<\eps$, $d(f^r(v_2),f^r(w_2))<\epsilon$.
Then $d(w_1,w_2)<\delta$ and $d(f^r(w_1),f^r(w_2))>\lambda-2\epsilon>4\eps$.
By the recurrence of $(w_1,w_2)$ there exists a positive integer $m>r$ such that
$d(w_1,f^m(w_1))<\delta/4$ and $d(w_2,f^m(w_2))<\delta/4$.
Define two finite sequences as follows
\begin{eqnarray*}
\eta(0)&=&(w_1,f(w_1),\dotsc,f^{m-1}(w_1)),\\
\eta(1)&=&(w_2,f(w_2),\dotsc,f^{m-1}(w_2)).
\end{eqnarray*}
Let $W_0=\overline{B}(f^r(w_1),\eps)$ and $W_1=\overline{B}(f^r(w_2),\eps)$.
Then $W_0$ and $W_1$ are  non-empty closed subsets, and $dist(W_0,W_1)>\epsilon$.

For every $\alpha=a_0a_1\dotsc a_n\dotsc\in\{0,1\}^{\N}$,
we set
\[Y_\alpha=\{x\in \overline{B}(u,2\eps): f^{mi}(x)\in\overline{B}(u,2\eps) \text{ and }
f^{mi+r}(x)\in W_{a_i}\text{ for }i=0,1,\dotsc\}.\]
It is clear that every $Y_\alpha$ is a closed subset of $X$, and $Y_\alpha\cap Y_\beta=\emptyset$ for $\alpha\neq\beta$.
We first show that every $Y_\alpha$ is not empty.
Let $y_\alpha$ be a point $\eps$-tracing the $\delta$-pseudo orbit $\eta(a_0)\eta(a_1)\dotsb\eta(a_n)\dotsb$.
Then it is easy to verify that $y_\alpha\in Y_\alpha$.

Let $Y=\bigcup_{\alpha\in\{0,1\}^{\N}}Y_\alpha$. Then $Y\subset U$.
We show that $(Y,f^m)$ is  a subsystem of $(X,f^m)$.
Let $\{y_n\}$ be a sequence in $Y$ and $y_n\to y$ as $n\to\infty$.
For every $y_n$, there exists $\alpha_n$ such that $y_n\in Y_{\alpha_n}$.
By the compactness of $\{0,1\}^{\N}$, without loss of generality we assume that
$\alpha_n\to \alpha\in\{0,1\}^{\N}$. Then $y\in Y_\alpha$, which implies that $Y$ is closed.
Now we show that $Y$ is $f^m$-invariant.
Let $y\in Y$. Then there exists an
$\alpha=a_0a_1\dotsc a_n\dotsc\in\{0,1\}^{\N}$ such that $y\in Y_\alpha$.
By the definition of $Y_\alpha$, $f^{mi}(y)\in\overline{B}(u,2\eps)$ and $f^{mi+r}(y)\in W_{a_i}$ for $i=0,1,\dotsc$.
Then $f^{mi}(f^m(y))\in\overline{B}(u,2\eps)$ and $f^{mi+r}(f^m(y))\in W_{a_{i}+1}$ for $i=0,1,\dotsc$,
that is $f^m(y)\in Y_{\sigma(\alpha)}\subset Y$.

Now define a map $\pi:Y\to \{0,1\}^{\N}$ by $\pi(Y_\alpha)=\{\alpha\}$ for every $\alpha\in\{0,1\}^{\N}$.
Then $\pi$ is a factor map between $(Y,f^m)$ and $(\{0,1\}^{\N},\sigma)$.
\end{proof}

\begin{rem}
Consider the map $f\colon [0,1]\to [0,1]$, $x \mapsto x^2$. It is well known that if an interval map has
fixed points only at endpoints then it has the shadowing property (e.g. see \cite[Lemma~4.1]{CL}).
Then $([0,1],f)$ has the shadowing property.
It is clear that $1$ is a sensitive point, but $\Omega(f)=\{0,1\}$,
and therefore we cannot apply Proposition~\ref{prop:sensitive-point} to $f$ in this case.
\end{rem}

Let $h(X,f)$ denote the topological entropy of $(X,f)$.
We refer the reader to the textbooks~\cite{Denker} or \cite{W82} for basic properties of topological entropy.
By Proposition~\ref{prop:sensitive-point}, we have the following dichotomy on non-wandering systems with the shadowing property.
Note that this result is also essentially contained  in \cite[Corollary 5(i)]{M11}.

\begin{thm}\label{thm:non-wandering}
Let $(X,f)$ be a non-wandering dynamical system with the shadowing property.
Then either $(X,f)$ is equicontinuous or $(X,f)$ has positive entropy.
\end{thm}

\begin{proof}
If  $(X,f)$ is not equicontinuous, then there exists a sensitive point in $X$.
Since $(X,f)$ is non-wandering, $\Omega(f)=X$. Then
by Proposition~\ref{prop:sensitive-point}, there exists a positive integer $m$,
a subsystem $(Y,f^m)$ of $(X,f^m)$  and a factor map $\pi:(Y,f^m)\to (\{0,1\}^{\N},\sigma)$.
Therefore, $h(X,f)=\frac{1}{m} h(X,f^m)\geq \frac{1}{m} h(\{0,1\}^{\N},\sigma)>0$.
\end{proof}

\begin{lem}\label{lem:factor_Rf}
Let $(X,f)$ be a dynamical system.
If there exists a positive integer $m$ and a subsystem $(Y,f^m)$ of $(X,f^m)$
such that $(Y,f^m)$ is an extension of $(\{0,1\}^{\N},\sigma)$,
then $(\Omega(f)\setminus R(f))\cap Y\neq \emptyset$, $(R(f)\setminus M(f))\cap Y\neq \emptyset$ and
 $(M(f)\setminus RR(f))\cap Y\neq \emptyset$.
\end{lem}
\begin{proof}
Let $\pi:(Y,f^m)\to (\{0,1\}^{\N},\sigma)$ be the factor map.
Using Kuratowski-Zorn Lemma there is a closed and $f^m$-invariant set $\Lambda\subset Y$ such that $\pi(\Lambda)=\set{0,1}^{\N}$ and if $Z\subset \Lambda$
is a closed and $f^m$-invariant subset such that $\pi(Z)=\{0,1\}^{\N}$ then $Z=\Lambda$.
Let $q\in \set{0,1}^{\N}$ be a point with dense orbit under $\sigma$ and fix $y\in \pi^{-1}(q)\cap \Lambda$.
There is an increasing sequence $\{n_i\}$ such that $\lim_{i\to \infty}\sigma^{n_i}(q)=q$ and we can also
assume, passing to a subsequence if necessary, that the limit $\lim_{i\to \infty} f^{mn_i}(y)=z$ exists.
But then if we put $Z=Orb(z,f^m)\subset \omega(y,f^m)$ then $q\in \pi(Z)$ and so $\pi(Z)=\set{0,1}^{\N}$.
It immediately implies that $y\in \omega(y,f^m)=\Lambda$, and hence
$(\Lambda,f^m)$ is transitive.
Since the orbit of $y$ is dense in $(\Lambda,f^m)$ and $(\set{0,1}^{\N},\sigma)$ is not minimal, we see that $y\in R(f)\setminus M(f)$.
If we take, $\alpha=1000\ldots \in \Omega(\sigma)\setminus R(\sigma)$,
then $\pi^{-1}(\alpha)\subset \Omega(\Lambda,f^m)\setminus R(f^m)\subset \Omega(f)\setminus R(f)$, since $R(f)=R(f^m)$.
Finally, if $(Z,\sigma)$ is a weakly mixing minimal system in $(\{0,1\}^{\N},\sigma)$ (e.g. a Chac\'{o}n flow)
and $z$ is any minimal point in $\pi^{-1}(Z)$ then $z\in M(f^m)\setminus RR(f^m)=M(f)\setminus RR(f)$.
\end{proof}

It is shown in~\cite{MO} that the set of regularly recurrent points is dense if
$(X,f)$ is a non-wandering system with the shadowing property.
If the system $(X,f)$ is also sensitive, then  $M(f)\setminus RR(f)$ is dense in $X$.
We can improve the conclusion as follows (Example~\ref{ex:dense_sp} below shows that while every sensitive system
has a dense set of sensitive points, the converse is not true).

\begin{prop}\label{prop:dense-sensitive-points}
Let $(X,f)$ be a  non-wandering system with the shadowing property.
If $(X,f)$ has a dense set of sensitive points, then sets $\Omega(f)\setminus R(f)$,
 $R(f)\setminus M(f)$ and
 $M(f)\setminus RR(f)$ are dense in $X$.
\end{prop}
\begin{proof}
Fix a non-empty open subset $U$ of $X$.
There exists a sensitive point $u\in U$.
By Proposition~\ref{prop:sensitive-point},
there exists a positive integer $m$,
a subsystem $(Y,f^m)$ of $(X,f^m)$ with $Y\subset U$
and a factor map $\pi\colon (Y,f^m)\to (\{0,1\}^{\N},\sigma)$.
Then the result follows by Lemma~\ref{lem:factor_Rf}.
\end{proof}

\begin{exmp}\label{ex:dense_sp}
Let $T\colon [0,1]\to [0,1]$ be the standard tent map, that is $T(x)=1-|1-2x|$.
Let
$$X=\{(0,0)\}\cup\bigcup_{k=1}^\infty \{\tfrac{1}{k}\}\times[0,\tfrac{1}{k}],$$
be endowed with metric induced by the Euclidean metric and put
$$
f(\tfrac{1}{k},x)=\begin{cases}
(\tfrac{1}{k},\tfrac{1}{k} T(kx)),& \text{ for }k>0\\
(0,0),& \text{ otherwise}
\end{cases}.
$$
Then this map has the shadowing property. Simply, if we fix $\eps>0$ then there
is $\delta>0$ such that if $(\tfrac{1}{k},x),(\tfrac{1}{s},y)\in X$ and $\tfrac{1}{k}>\sqrt{\eps}$
then either $k=s$ or $|\tfrac{1}{k}-\tfrac{1}{s}|>\delta$. This immediately implies that for any $\delta$-pseudo-orbit
$\xi=\set{z_n}_{n=0}^\infty=\{(\tfrac{1}{k_n},x_n)\}_{n=0}^\infty$ we either have that $\xi \subset [0,\sqrt{\eps}]\times [0,\sqrt{\eps}]$
or $k_n=k_{n+1}$ for every $n\geq 0$. But in the first case $(0,0)$ is an $\eps$-tracing point for $\xi$
and in the second case we can use the shadowing property of $T$ (after appropriate rescaling of the sequence $\set{x_n}_{n=0}^\infty$).
This proves that $f$ has the shadowing property but it is also clear that $f$ is equicontinuous at $(0,0)$.
\end{exmp}

Now we can provide a series of conditions equivalent, for systems with the shadowing property,
to positive entropy.

\begin{thm}\label{thm:shadowing-positive-entropy}
Let $(X,f)$ be a dynamical system with the shadowing property.
Then the following conditions are equivalent:
\begin{enumerate}
  \item\label{spe:c1} $(X,f)$ has positive entropy;
  \item\label{spe:c2} there exists a sensitive transitive subsystem $(M,f)$ of $(X,f)$;
  \item\label{spe:c3} there exists a sensitive point in $(\Omega(f),f)$;
  \item\label{spe:c4} there exists a positive integer $m$, a subsystem $(Y,f^m)$ for $f^m$
   and a factor map $\pi:(Y,f^m)\to (\{0,1\}^{\N},\sigma)$;
  \item\label{spe:c5} $\Omega(f)\setminus R(f)$ is not empty;
  \item\label{spe:c6} $R(f)\setminus M(f)$ is not empty;
  \item\label{spe:c7} $M(f)\setminus RR(f)$ is not empty.
\end{enumerate}
\end{thm}
\begin{proof}
$\eqref{spe:c1}\Rightarrow\eqref{spe:c2}$ By the variation principle of topological entropy,
there exists an ergodic invariant Borel probability measure $\mu$ on $X$
such that the measure-theoretic entropy of $\mu$ is positive.
Denote by $M$ the support of $\mu$.
By~\cite{GW93}, $(M,f)$ is either sensitive or equicontinuous
(it is so called E-system, i.e. transitive map with fully supported measure).
Then the only possibility is that $(M,f)$ is sensitive, since it has positive entropy.

$\eqref{spe:c2}\Rightarrow\eqref{spe:c3}$ Let $x\in M$. Then $x$ is a sensitive point of $(M,f)$.
Since $M\subset \Omega(f)$, $x$ is also a sensitive point of $(\Omega(f),f)$.

$\eqref{spe:c3}\Rightarrow\eqref{spe:c4}$ follows from Proposition~\ref{prop:sensitive-point}.

$\eqref{spe:c4}\Rightarrow\eqref{spe:c1}$ follows by $h(X,f)=\frac{1}{m} h(X,f^m)\geq \frac{1}{m} h(\{0,1\}^{\N},\sigma)>0$.

$\eqref{spe:c4}\Rightarrow (\ref{spe:c5},\ref{spe:c6},\ref{spe:c7})$
Applying Lemma~\ref{lem:factor_Rf},
we obtain that $\Omega(f)\setminus R(f)$, $R(f)\setminus M(f)$, and $M(f)\setminus RR(f)$ are non-empty.

Finally, observe that if $(X,f)$ has zero topological entropy, then $(\Omega(f),f)$ also has zero entropy by the variational principle.
By Theorem~\ref{thm:non-wandering}, $(\Omega(f),f)$ must be equicontinuous since $(\Omega(f),f)$ also has the shadowing property by Theorem~\ref{thm:shad:omega}.
In particular, in this situation we have $\Omega(f)=R(f)=M(f)=RR(f)$ by Proposition~\ref{prop:eq-RR}, which proves
$\eqref{spe:c5}\Rightarrow \eqref{spe:c1}$, $\eqref{spe:c6}\Rightarrow \eqref{spe:c1}$ and $\eqref{spe:c7}\Rightarrow \eqref{spe:c1}$
completing the proof.
\end{proof}

\section{Weakly mixing systems with the shadowing property}
Recall that a dynamical $(X,f)$ is called \emph{weakly mixing} if $(X\times X,f\times f)$ is transitive.
By the well known  Furstenberg Intersection Lemma, if $(X,f)$ is weakly mixing,
then $(X^n,f^{(n)})$ is transitive for all $n\in\mathbb{N}$, where $X^n=X\times X\times \dotsb \times X$ ($n$-times)
and $f^{(n)}=f\times f\times \dotsb \times f$ ($n$-times).

It is shown in~\cite{L11} that for a continuous map $f\colon [0,1]\to [0,1]$
weak mixing is equivalent to uniformly positive entropy of all orders.
The main result of this section is that the similar result also holds for non-trivial systems with the shadowing property.

\begin{defn}[\cite{B92,HY06}]
A dynamical system $(X,f)$ is said to have \emph{uniformly positive entropy of all orders}
if for every $n\geq 2$, any cover of $X$ by $n$ non-dense open sets has positive entropy.
\end{defn}
A system with uniformly positive entropy of all orders is a topological analogue
of the Kolmogorov system in ergodic theory, so such a system is also called a \emph{topological K system}.

\begin{defn}[\cite{B92,HY06}]
A dynamical system $(X,f)$ is said to have the \emph{strong Property P}
if for any $n\geq 2$ and any non-empty open subsets $U_0,U_1,\dotsc,U_{n-1}$ of $X$
there exists an integer $N$ such that whatever $k\geq 2$,
whatever $s=(s(1),s(2),\dotsc,s(k))\in\{0,1,\dots,n-1\}^k$,
there exists $x\in X$ with $x\in U_{s(1)},f^N(x)\in U_{s(2)}\dotsc, f^{(k-1)N}x\in U_{s(k)}$.
\end{defn}

For $i\in \set{0,1,\ldots, d}$, denote by $C[i]$ the \emph{cylinder set} defined by $i$,
that is
$$C[i]=\set{x \in \{0,1,\dotsc,d\}^{\N} : x_0=i}.$$

\begin{thm}\label{prop:weak-mixing}
Let $(X,f)$ be a non-trivial non-wandering system with the shadowing property.
Then the following conditions are equivalents:
\begin{enumerate}
  \item\label{equiv:wmpte:c1} $(X,f)$ is weakly mixing;
  \item\label{equiv:wmpte:c2} For every $d\geq 1$ and every non-empty open subsets $U_0,U_1,\dotsc,U_d$ of $X$,
there exists a positive integer $m\in\mathbb{N}$, a subsystem $(Y,f^m)$ of $(X,f^m)$ and
 a factor map $\pi:(Y,f^m)\to (\{0,1,\dotsc,d\}^{\N},\sigma)$
 such that $\pi^{-1}(C[i])\subset U_i$ for $i=0,1,\dotsc,d$;
 \item\label{equiv:wmpte:c3} $(X,f)$ has the strong Property P;
  \item\label{equiv:wmpte:c4} $(X,f)$ has uniformly positive entropy of all orders.
\end{enumerate}
\end{thm}
\begin{proof}
  $\eqref{equiv:wmpte:c1}\Rightarrow\eqref{equiv:wmpte:c2}$
  Choose points $z_i\in U_i$ for $i=0,1,\dotsc,d$ such that $z_i \neq z_j$ for $i\neq j$.
  Let $\lambda<\min_{i\neq j} {d(z_i,z_j)}$ with $\overline{B}(z_i,\lambda)\subset U_i$ for $i=0,1,\dotsc,d$.
  Let $\eps<\frac{1}{8}\lambda$ and let $\delta>0$ such that every $\delta$-pseudo orbit of
  $f$ is $\eps$-traced by some point in $X$.
  Choose a transitive point $(u_0,u_1,\dotsc,u_d)$ in $(X^{d+1}, f^{(d+1)})$ with
  $\max_{i\neq j} d(u_i,u_j)<\frac{1}{4}\delta$.
  Then there exists a positive integer $r$ such that
  $d(f^r(u_i),z_i)<\frac{1}{4} \eps$ for $i=0,1,\dotsc,d$, and a positive integer $m>r$
  such that $d(f^m(u_i),u_i)<\frac{1}{4} \delta$ for $i=0,1,\dotsc,d$.

Define $d+1$ finite sequences as follows
\begin{align*}
  \eta(0)&=(u_0,f(u_0),\dotsc,f^{m-1}(u_0)),\\
  \eta(1)&=(u_1,f(u_1),\dotsc,f^{m-1}(u_1)),\\
 & \dotsc\\
  \eta(d)&=(u_d,f(u_d),\dotsc,f^{m-1}(u_d)).
\end{align*}

Let $W_0=\overline{B}(f^r(u_0),\eps)$, $W_1=\overline{B}(f^r(u_1),\eps)$,
$\dotsc$, $W_d=\overline{B}(f^r(u_d),\eps)$.
Then $W_0$, $W_1$, $\dotsc$, $W_d$ are  non-empty closed subsets, and $dist(W_i,W_j)>\eps$ for $i\neq j$.
Clearly $W_i\subset \overline{B}(z_1,2\eps)\subset \overline{B}(z_1,\lambda)\subset U_i$.

If $\alpha=a_0a_1\dotsc a_n\dotsc\in\{0,1,\dotsc,d\}^{\N}$,
we set $Y_\alpha=\{x\in X: f^{mi}(x)\in W_{a_i}\text{ for }i=0,1,\dotsc\}$.
Then $Y_\alpha$ is a closed subset of $X$. We should show that $Y_\alpha$ is not empty.
Let $y_\alpha$ be a point $\eps$-tracing the $\delta$-pseudo orbit $\eta(a_0)\eta(a_1)\dotsb\eta(a_n)\dotsb$.
Then $f^{r}(y_\alpha)\in Y_\alpha$. This shows that $Y_\alpha$ is not empty.
Let $Y=\bigcup_{\alpha\in\{0,1,\dotsc,d\}^{\N}}Y_\alpha$
and let $\pi:Y\to \{0,1,\dotsc,d\}^{\N}$, $\pi(Y_\alpha)=\alpha$.
Then $Y$ is closed and $\pi$ is a factor map between $(Y,f^m)$ and $(\{0,1,\dotsc,d\}^{\N},\sigma)$.
Finally, note that $\pi^{-1}(C[i])=\bigcup_{\alpha\in\{0,1,\dotsc,d\}^{\N}}Y_{i\alpha}\subset U_i$
for $i=0,1,\ldots,d$.

$\eqref{equiv:wmpte:c2}\Rightarrow\eqref{equiv:wmpte:c3}$
Fix $n\geq 2$ and non-empty open subsets $U_0,U_1,\dotsc,U_{n-1}$ of $X$.
Then there exists a positive integer $m\in\mathbb{N}$ and a subsystem $(Y,f^m)$ and
 a factor map $\pi:(Y,f^m)\to (\{0,1,\dotsc,n-1\}^{\N},\sigma)$
 such that $\pi^{-1}(C[i])\subset U_i$ for $i=0,1,\dotsc,n-1$.

For each $k\geq 2$,
and $s=(s(1),s(2),\dotsc,s(k))\in\{0,1,\dots,n-1\}^k$,
if $z\in \{0,1,\dotsc,n-1\}^{\N}$ has $s$ as its prefix,
then for each $x\in\pi^{-1}(z)$ we have
$x\in \pi^{-1}(C[s(1)])\subset U_{s(1)}, f^{mM}(x)\in\pi^{-1}(C[s(2)])\subset U_{s(2)},\dotsc,
f^{(k-1)mM}x\in \pi^{-1}(C[s(k)])\subset U_{s(k)}$.

$\eqref{equiv:wmpte:c3}\Rightarrow\eqref{equiv:wmpte:c4}$ follows from~\cite[Theorem~7.4]{HY06}.

$\eqref{equiv:wmpte:c4}\Rightarrow\eqref{equiv:wmpte:c1}$ follows from~\cite[Propsition~2]{B92}.
\end{proof}

\begin{defn}
A dynamical system $(X,f)$ is \emph{positively expansive (with an expansive constant $\beta>0$)} if
for any $x,y \in X$ with $x \neq y$ there is $n > 0$ such that $d(f^n(x),f^n(y)) >\beta$.
\end{defn}

\begin{rem}\label{rem:positive-expansive-conjugacy}
If a dynamical system $(X,f)$ is positively expansive with expansive constant $\beta$,
then for every pseudo-orbit $\xi$ there is at most one point $(\beta/2)$-tracing $\xi$.
Then it is easy to see that in the proof of Theorem~\ref{prop:weak-mixing},
every $Y_\alpha$ is a singleton, and as a consequence $\pi$ is a conjugacy.
\end{rem}

The following fact highlights an important property of (positively) expansive systems
(see \cite[Theorem~3.4.4.]{AH94}).
\begin{thm}[Topological Decomposition Theorem]
Let $(X,f)$ be a positively expansive dynamical system with the shadowing property and assume that $f$ surjective.
Then the following properties hold:
\begin{enumerate}
\item $\Omega(f)$ contains a finite
sequence of pairwise disjoint $f$-invariant closed subsets $B_i$ ($1 \leq i \leq l$)
such that $\Omega(f)=\bigcup_{i=1}^l B_i$ and
each subsystem $(B_i,f)$ is transitive (sets $B_i$ are called \emph{basic sets}).
\item For every basic set $B$ there is an integer $a>0$ and a finite
sequence of pairwise disjoint $f^a$-invariant closed sets $C_i\subset B$ ($0 \leq i < a$) such that $B=\bigcup_{i=0}^{a-1} C_i$, $f(C_i)=C_{i+1\pmod{a}}$ and
each subsystem $(C_i,f^a)$ is strongly mixing (sets $C_i$ are called \emph{elementary sets}).
\end{enumerate}
\end{thm}

\begin{thm}\label{thm:shadowing-expansive}
Let $(X,f)$ be a non-wandering system with the shadowing property, where $X$ has no isolated points.
If $(X,f)$ is positively expansive, then for every $d\geq 1$ and every non-empty open subset $U$ of $X$,
there exists a positive integer $m\in\mathbb{N}$ and a subsystem $(Y,f^m)$ with $Y\subset U$ and
a conjugacy map $\pi:(Y,f^m)\to (\{0,1,\dotsc,d\}^{\N},\sigma)$.
\end{thm}
\begin{proof}
Let $U$ be a non-empty open subset of $X$.
By the Topological Decomposition Theorem, there is an elementary set $C$ such that $U\cap C\neq\emptyset$.
Since $X$ has no isolated points, so does $C$.
There exists  $a>0$ such that $(C,f^a)$ is strongly mixing.
It is easy to see that $(C,f^a)$ is also positively expansive and has the shadowing property.
Then the result follows from Theorem~\ref{prop:weak-mixing} and Remark~\ref{rem:positive-expansive-conjugacy}.
\end{proof}

\begin{rem}
In Theorem~\ref{thm:shadowing-expansive}, we assume that $X$ has no isolated points.
If fact, for every non-wandering dynamical system $(X,f)$,
we easily obtain that $X$ has at most countably many isolated points
and every isolated point in $X$ must be periodic.
\end{rem}

\section{Specification property}

Let $(X,f)$ be a dynamical system. We say that $f$ satisfies the \emph{periodic specification property}
if for any $\eps >0$ there exists $M>0$ such that
for any $k \geq 2$, any $k$ points $x_1, x_2, \ldots , x_k \in X$, any non-negative integers $0\leq a_1 \leq b_1 < a_2 \leq b_2 < \ldots < a_k \leq b_k$
with $a_i - b_{i-1} \geq M$ for each $i = 2, 3, \ldots , k$ and any integer $p\geq  M + b_k - a_1$,
there exists a periodic point $z\in X$ with $f^p(z)=z$ and
$d(f^{j}(z), f^j(x_i )) < \eps$ for all $a_i\leq  j \leq b_i$ and $1 \leq i \leq k$.
We say that $f$ satisfies the \emph{specification property} if
the point $z$ in the periodic specification property is not requested to be periodic (hence no condition on $p$).

It was proved first by Bowen (see \cite{Denker}) that
if $(X,f)$ is a weakly mixing system with the shadowing property,
then it satisfies the specification property.
Moreover, if $(X,f)$ is also positively expansive,
then it satisfies the periodic specification property.
It is shown in~\cite{MO} that every non-wandering system with the shadowing property
has a dense set of regularly recurrent points.
We combine these above results and show that every weak mixing system with the shadowing property
also has the following version of specification property.
The proof is inspired by Lemma~3.1 in \cite{MO}.

\begin{thm}\label{thm:spec}
Let $(X,f)$ be a weakly mixing system with the shadowing property. For any $\eps >0$ there exists $M>0$ such that
for any $k \geq 2$, any $k$ points $x_1, x_2, \ldots , x_k \in X$,
 any non-negative integers $0\leq a_1 \leq b_1 < a_2 \leq b_2 < \ldots < a_k \leq b_k$
with $a_i - b_{i-1} \geq M$ for each $i = 2, 3, \ldots , k$ and  any $p \geq M + b_k-a_1$, there exists $z \in RR(f)$
such that $d(f^{j}(z),f^{np+j}(z)) < \eps$ for every $n,j\geq 0$ and
$d(f^{np+j}(z), f^j(x_i )) < \eps$ for all $a_i\leq  j \leq b_i$, $1 \leq i \leq k$ and $n\geq 0$.
\end{thm}
\begin{proof}
First note that we may assume that $a_1=0$. Simply, if $z'$ is a point obtained for $a_i'=a_i-a_1$, $b_i'=b_i-a_1$
and $x_i'=f^{a_1}(x_i)$ then $z=f^{p-a_1}(z')$ is a desired point.

Let $\lambda_0=\eps/8$ and let $\delta>0$ be such that every $\delta$-pseudo-orbit is $\lambda_{0}/8$-traced.
Since $f$ is weakly mixing, by results of \cite{RW} it is chain mixing, that is,
there is $M>0$ such that for any two points $x,y\in X$ and any $n\geq M$ there is a $\delta$-chain
of length $n$ from $x$ to $y$.

Now fix $k$, $p$ and sequences $x_i$, $a_i$, $b_i$ as in the statement of theorem (with $M$ fixed above).
There is a periodic $\delta$-pseudo-orbit
$$
\xi = (z_0, \ldots, z_{p-1}, z_0, \ldots z_{p-1}, z_0, \ldots)
$$
such that $z_{j}=f^j(x_i)$ for $i=1,\ldots,k$ and all $a_i\leq  j \leq b_i$.
Let $q_{0}$ be a point which is $\lambda_0/8$-tracing $\xi$.
Let $y_{0}$ be any minimal point in the set $\omega(q_{0},f^{p})$.
For any $i\geq 0$ there is $\tau>0$ such that $d(f^{i}(y_{0}),f^{\tau p+i}(q_{0}))<\lambda_0/{8}$.
Then
\begin{eqnarray*}
d(f^i(y_0),\xi_i)&\leq& d(f^{i}(y_{0}),f^{\tau p+i}(q_{0}))+d(\xi_i,f^{\tau p+i}(q_{0}))
< \lambda_0/{8}+\lambda_0/{8}\\
&\leq& \lambda_0/4,
\end{eqnarray*}
which implies that $y_0$ is $\lambda_0/4$-tracing $\xi$.

Denote $\lambda_r=\eps/8^{r+1}$ for $r=0,1,2,\ldots$ and $m_0=p$.
Now we will apply arguments inspired by the proof of Lemma~3.1 in \cite{MO},
and construct an increasing sequence $\set{m_r}_{r=0}^\infty$ of natural numbers and
a sequence $\set{y_r}_{r=0}^\infty$ of points of $X$ such that for every $r=0,1,2,\ldots$ we have
\begin{enumerate}
\item\label{spec:rr:c1} $m_{r}$ divides $m_{r+1}$;
\item\label{spec:rr:c2} $y_r$ is a minimal point and $d(f^{i m_r+j}(y_r), f^j(y_r))<\frac{\lambda_{r+1}}{2}$
for any $i,j\geq 0$;
\item\label{spec:rr:c3} $d(f^{j}(y_r),f^{j}(y_{l-1})) \leq \sum_{i=l}^r \lambda _i$ for $1\le l\le r$ and every $j \geq 0$.
\end{enumerate}

Before we proceed with the construction, let us first assume that
we have already constructed a sequence of points  $\{y_r\}$ satisfying \eqref{spec:rr:c1}--\eqref{spec:rr:c3}.
Let $z$ be a limit point of the sequence $\{y_r\}$.
For every $j\geq 0$ there is $s>0$ such that $d(f^j(z),f^j(y_{i_s}))<\eps/8$ and
then (applying condition \eqref{spec:rr:c3})
we obtain that
\begin{eqnarray*}
d(f^j(z),\xi_j)&\leq &d(\xi_j,f^j(y_0))+d(f^j(y_0), f^j(y_{i_s}))+d(f^j(y_{i_s}),f^j(z))\\
&\leq & \frac{\eps}{8}+\sum_{i=1}^{i_s} \lambda_i+\frac{\eps}{8}\leq \frac{\eps}{4}+\sum_{i=1}^\infty \lambda_i < \frac{\eps}{2}.
\end{eqnarray*}
In particular, for any $n\geq 0$ we have
\begin{eqnarray*}
d(f^j(z),f^{np+j}(z))&\leq &d(f^j(z),\xi_j)+d(\xi_j,f^{np +j}(z))\\
&=&d(f^j(z),\xi_j)+d(\xi_{np +j},f^{np+j}(z))\\
&<&\eps.
\end{eqnarray*}

Similarly, for every $r\geq 1$ and every $j\geq 0$
there is $s>r$ such that $d(f^{jm_r}(z),f^{jm_r}(y_{i_s}))<\lambda_r$ and $d(z,y_{i_s})<\lambda_r$.
Clearly $i_s\geq s$.
Additionally, observe that \eqref{spec:rr:c3} implies
\begin{eqnarray*}
d(y_{i_s},y_{r})) &\leq& \sum_{i=r+1}^{i_s} \lambda _i,\\
d(f^{jm_r}(y_{i_s}),f^{jm_r}(y_{r}))) &\leq& \sum_{i=r+1}^{i_s} \lambda _i
\end{eqnarray*}
and by \eqref{spec:rr:c2} we also have that
$d(f^{jm_r}(y_{r}),y_r)<\frac{\lambda_{r+1}}{2}$
which gives the following
\begin{eqnarray*}
d(f^{jm_r}(z),z)&\leq & d(f^{jm_r}(z),f^{jm_r}(y_{i_s}))+d(f^{jm_r}(y_{i_s}),f^{jm_r}(y_{r}))\\
&&\quad+d(f^{jm_r}(y_{r}),y_{r})+d(y_r, y_{i_s})+d(y_{i_s},z)\\
&\leq& \lambda_r +2\sum_{i={r+1}}^{i_s} \lambda _i+\frac{\lambda_{r+1}}{2}+\lambda_r \leq 3\sum_{i=r}^\infty \lambda_i.
\end{eqnarray*}
But $\lim_{r\to\infty}\sum_{i=r}^\infty \lambda_i=0$, hence $z\in RR(f)$.

To complete the proof, we need to perform a construction of the sequences $\{m_r\}$ and $\{y_r\}$ satisfying \eqref{spec:rr:c1}--\eqref{spec:rr:c3}.
Fix $s\geq 0$ and suppose that $m_r,y_r$ are constructed for every $0\leq r \leq s$.
We will show, how $m_{s+1}$ and $y_{s+1}$ can be constructed.
Let $\alpha>0$ be such that every $\alpha$-pseudo-orbit is $\lambda_{s+2}/8$-traced.
Since $y_s$ is minimal for $f$, it is also minimal for $f^{m_s}$, in particular, there is
a positive integer $m_{s+1}$  divisible by $m_s$
such that $d(f^{m_{s+1}}(y_s),y_s)<\alpha$. Take a periodic $\alpha$-pseudo-orbit
$$
\gamma=(y_s,f(y_s),\ldots, f^{m_{s+1}-1}(y_s),y_s,f(y_s),\ldots, f^{m_{s+1}-1}(y_s),y_s,\ldots)
$$
and let $q_{s+1}$ be a point which is $\lambda_{s+2}/8$-tracing $\gamma$.
Let $y_{s+1}$ be any minimal point in the set $\omega(q_{s+1},f^{m_{s+1}})$.
If we fix any $i,j\geq 0$ then there is $\tau>0$
such that $d(f^{i m_{s+1}+j}(y_{s+1}),f^{(i+\tau) m_{s+1}+j}(q_{s+1}))<\lambda_{s+2}/8$
and $d(f^{j}(y_{s+1}),f^{\tau m_{s+1}+j}(q_{s+1}))<\lambda_{s+2}/8$.
But then
\begin{eqnarray*}
d(f^{i m_{s+1}+j}(y_{s+1}), f^j(y_{s+1}))&<& d(f^{i m_{s+1}+j}(y_{s+1}),\gamma_{\tau m_{s+1}+j})+d(\gamma_{\tau m_{s+1}+j}, f^j(y_{s+1}))\\
&= & d(f^{i m_{s+1}+j}(y_{s+1}),\gamma_{(i+\tau)m_{s+1}+j})+d(\gamma_{\tau m_{s+1}+j}, f^j(y_{s+1}))\\
&<& \frac{\lambda_{s+2}}{8}+\frac{\lambda_{s+2}}{8}+\frac{\lambda_{s+2}}{8}+\frac{\lambda_{s+2}}{8}= \frac{\lambda_{s+2}}{2},
\end{eqnarray*}
which proves \eqref{spec:rr:c2}. Similarly, fix any $1\le l\le s+1$ and any $j \geq 0$,
there is $\tau>0$
such that $d(f^{j}(y_{s+1}),f^{\tau m_{s+1}+j}(q_{s+1}))<\lambda_{s+2}/8$.
There is $0\leq t<m_{s+1}$ and $i\geq 0$ such that $j=i m_{s+1}+t$ and then
\begin{eqnarray*}
d(f^{j}(y_{s+1}),f^{j}(y_{l-1})) &\leq& d(f^{j}(y_{s+1}),\gamma_{\tau m_{s+1}+j}) + d(\gamma_{\tau m_{s+1}+j}, f^{j}(y_{l-1}))\\
&\leq & d(f^{j}(y_{s+1}),\gamma_{\tau m_{s+1}+j}) + d(\gamma_t, f^j(y_s)) + d(f^j(y_s), f^{j}(y_{l-1}))\\
&\leq& \frac{\lambda_{s+2}}{8}+\frac{\lambda_{s+2}}{8} + d(f^t(y_s), f^{im_{s+1}+t}(y_s))+\sum_{i=l}^s \lambda _i\\
&\leq& \frac{\lambda_{s+2}}{4} + \frac{\lambda_{s+1}}{2} +\sum_{i=l}^s \lambda _i<\sum_{i=l}^{s+1} \lambda _i,
\end{eqnarray*}
which proves \eqref{spec:rr:c3} and ends the proof.
\end{proof}

As an easy consequence of Theorem~\ref{thm:spec} we obtain a classical result, proved first by Bowen (see \cite{Denker}).

\begin{cor}\label{cor:expansitive-psp}
If $(X,f)$ is weakly mixing, positively expansive and has the shadowing property,
then it has the periodic specification property.
\end{cor}
\begin{proof}
Fix $\eps<\beta$ where $\beta$ is an expansive constant. The only
thing which we need to extend in Theorem~\ref{thm:spec}
to obtain periodic specification property is that $z$ is a periodic point with period $p$.
But if $\eps<\beta$ then in Theorem~\ref{thm:spec}
for every $j\geq 0$ we have
$$
d(f^j(z),f^{j+p}(z))\leq \eps<\beta.
$$
Hence $z=f^p(z)$ completing the proof.
\end{proof}

Given a compact metric space $X$, denote by $C(X,\R)$ the set of all continuous functions $\xi \colon X\to \R$.
If we endow it with the norm $\Vert\xi\Vert=\sup_{x\in X}|\xi(x)|$ then it becomes a Banach space.

Let $\mathcal{M}(X)$ be the set of all Borel probability measures on $X$ and
fix a dense sequence $\set{\xi_i}_{i=1}^\infty\subset C(X,\R)$.
Then the metric
$$
\mathcal{D}(\mu,\nu)=\sum_{i=1}^\infty\frac{|\int_X \xi_i d\mu-\int_X \xi_i d\nu|}{2^i (\Vert\xi_i\Vert+1)}
$$
for $\mu,\nu \in \mathcal{M}(X)$ is compatible with the weak-$^*$ topology on $\mathcal{M}(X)$.
We denote by $\mathcal{M}_f(X)\subset \mathcal{M}(X)$ the set of all invariant measures for $(X,f)$.
The \emph{support} of a measure $\mu\in M(X)$, denoted by $\supp(\mu)$,
is the smallest closed subset $C$ of $X$ such that $\mu(C)=1$.
If $\mu$ is an invariant measure for $(X,f)$, then it is clear that $\supp(\mu)$ is $f$-invariant.

Let $x\in X$ be a periodic point with periodic $p$. Then it corresponds to an invariant measure $\mu_x$
which has mass $1/p$ as each of the points $x, f(x), \dotsc,f^{p-1}(x)$.
We denote the set of these measures by $P(p)$. The following fact was first proved by Sigmund (see \cite{S70}).

\begin{thm}\label{thm:sieg}
If $(X,f)$ satisfies the periodic specification property and if $\ell\in\mathbb{N}$,
then $\bigcup_{p\geq\ell}P(p)$ is dense in $\mathcal{M}_f(X)$.
\end{thm}

While there are known examples of weakly mixing systems with the shadowing property that have no periodic points
(e.g. see Example~\ref{exmp:shadowing-non-period}), we can prove that the assertion of Theorem~\ref{thm:sieg}
can be preserved to some extent.
Strictly speaking, we can show that if a weakly mixing system has the shadowing property,
then the set of ergodic invariant measures supported on the closures of orbits of regularly recurrent points
is dense in the set of invariant measures.

\begin{defn}
Denote by $E^T_f(X)$ the set of ergodic measures $\mu\in \mathcal{M}_f(X)$ for $(X,f)$
such that the support of $\mu$ is the closure of the orbit of a regularly recurrent point.
\end{defn}

Denote by $Q(f)$ the set of quasi-regular points with respect to $f$, that is the set of points $x$ such that
the limit
$$
\xi^*(x)=\lim_{n\to \infty}\frac{1}{n}\sum_{j=0}^{n-1}\xi(f^j(x))
$$
exists for every $\xi \in C(X,\R)$.
It can be proved that $Q(f)$ is a Borel set and that $\mu(Q(f))=1$ for every $\mu\in \mathcal{M}_f(X)$ (e.g. see \cite{AH94}).

The proof of the following fact is in main part the same as original
Sigmund's argument in the proof of Theorem~\ref{thm:sieg}. The final argument is made by application
of Theorem~\ref{thm:spec}. Since it could be hard to present rigorous explanation of the proof
without detailed reference to \cite{S70}, we decided to provide a complete proof.
It makes the paper complete (and hence accessible to the reader), and at the same
time fits into approach presented by various authors before,
when proving variants of Sigmund's result (e.g. see \cite{FH}).

\begin{cor}\label{thm:shadowing-ETf}
Let $(X,f)$ be a weakly mixing system with the shadowing property. Then $E^T_f(X)$ is a dense subset of $\mathcal{M}_f(X)$.
\end{cor}
\begin{proof}
Fix any $\mu\in \mathcal{M}_f(X)$ and its open neighborhood $U$. By the definition of metric $\mathcal{D}$ there is $\eps>0$
and a finite set $F\subset C(X,\R)$, such that $\Vert\xi\Vert\leq 1$ for all $\xi \in F$ and
$$
W=\set{\nu \in \mathcal{M}_f(X) : \left|\int_X \xi d\mu - \int_X \xi d\nu\right|<\eps \text{ for all } \xi \in F}\subset U.
$$
By Birkhoff Ergodic Theorem and the fact that $\mu(Q(f))=1$ we have for each $\xi \in F$
$$
\int_{Q(f)}\xi d\mu=\int_{Q(f)}\xi^* d\mu.
$$
It is clear that $\xi^*|_{Q(f)}$ is Borel and $\sup_{x\in Q(f)}|\xi^*(x)|\leq 1$ for all $\xi\in F$.
Let $\mathcal{P}=\set{P_1,\ldots, P_s}$ be a partition of $Q(f)$ into non-empty Borel sets
such that $\xi^*|_{P_i}$ has oscillation bounded by $\eps/4$ for all $\xi\in F$ and $i=1,\dotsc,s$
(i.e. $\sup_{x\in P_i} \xi^*(x) - \inf_{y\in P_i} \xi^*(y)\leq \eps/4$).

For each $j=1,\ldots,s$, choose a point $y_j\in P_s$
and observe that
$$
\left|\int_{P_j}\xi^* d\mu - \mu(P_j)\xi^*(y_j) \right|\leq \frac{\eps}{4}\mu(P_j)
$$
which immediately implies that
$$
\left|\int_{Q(f)}\xi d\mu - \sum_{j=1}^s\mu(P_j)\xi^*(y_j)\right|\leq \frac{\eps}{4}.
$$

There is $\delta>0$ such that if $d(z_1,z_2)<\delta$ then $|\xi(z_1)-\xi(z_2)|<\eps/8$, provided that $\xi \in F$
and $z_1,z_2\in X$.
Let $M$ be provided by Theorem~\ref{thm:spec} for $\delta/4$.

Directly from the definition of $Q(f)$ we can find $N>0$ such that for any $\xi \in F$ and $j=1,\ldots, s$ we have
$$
\left|\frac{1}{N}\sum_{i=0}^{N-1}\xi(f^i(y_j))-\xi^*(y_j)\right|<\frac{\eps}{8}.
$$
In particular we have
\begin{equation}
\left|\int_{Q(f)}\xi d\mu - \frac{1}{N}\sum_{j=1}^s\sum_{i=0}^{N-1}\mu(P_j)\xi(f^i(y_j))\right|\leq \frac{3\eps}{8}.
\label{eq:approxPj}
\end{equation}

Fix a positive integer $m$ such that $\frac{1}{m}< \frac{\eps}{8s}$.
There exist positive integers $n_1,\ldots, n_s$ such that for $j=1,\ldots,s$ we have
$$
\frac{n_j}{m}\leq \mu(P_j)<\frac{n_j+1}{m}.
$$
Putting $m_j=n_j$ or $m_j=n_j+1$ we can present $\sum_{j=1}^s m_j=m$
and clearly we also have
$$
\left|\mu(P_j)-\frac{m_j}{m}\right|\leq \frac{1}{m}\leq \frac{\eps}{8s}.
$$
Therefore
$$
\left|\frac{1}{N}\sum_{j=1}^s\sum_{i=0}^{N-1}\left(\mu(P_j)-\frac{m_j}{m}\right)\xi(f^i(y_j))\right|\leq \frac{\eps}{8N}\sum_{i=0}^{N-1}|\xi(f^i(y_j))|\leq \frac{\eps}{8}
$$
which combined with \eqref{eq:approxPj} implies that
\begin{equation}
\label{eg:almostCalc}\left|\int_{Q(f)}\xi d\mu-
\frac 1 N\sum_{j=1}^s\sum_{i=0}^{N-1}\frac{m_j}{m}\xi(f^i(y_j))\right|\leq \frac{\eps}{2}.
\end{equation}
Increasing $N$ if necessary, we may assume that $2M/N<\eps/8$.

Since $f$ is weakly mixing, it is onto and hence for every $a>0$ and $j$ there is $x_j$ such that $f^{a}(x_j)=y_j$.
For $r=1,\ldots,m$ we put $a_r=(r-1)(N+M)$, $b_r=a_r+N-1$
and let $x_r$ be such that $f^{a_r}(x_r)=y_j$ where $1\leq j\leq s$ is the largest number
such that $\sum_{i=1}^{j-1}m_j\leq r-1$. Put $p=m(N+M)$. Then
\begin{equation}\label{eg:almostCalc2}
\left|\frac{1}{p}\sum_{r=1}^{m}\sum_{i=a_r}^{b_r}\xi(f^i(x_r))-
\frac 1 N\sum_{j=1}^s\sum_{i=0}^{N-1}\frac{m_j}{m}\xi(f^i(y_j))\right|\leq \frac{M}{N+M}<\frac{\eps}{16}.
\end{equation}

Let $z$ be provided by Theorem~\ref{thm:spec} for $\delta/4$, the sequences $a_r$, $b_r$,
points $x_r$ and $p$.
Put $\Lambda=\omega(z,f)$ and take an ergodic measure $\nu$ on $\Lambda$.
Since $\nu$ is ergodic, there exists a point $q\in \Lambda$
such that
\[\lim_{n\to \infty} \frac{1}{n}\sum_{i=0}^{n-1}\xi (f^i(q))=\int_\Lambda \xi d\nu=\int_{Q(f)} \xi d\nu\]
for each $\xi \in F$.
Taking forward iteration of $q$ if necessary,
we may assume that $q\in \omega(z,f^p)$.
For every $j\geq 0$, there exists $\tau>0$ such that $d(f^{j}p, f^{\tau p+j}(z))<\delta/4$, and then
\begin{eqnarray*}
d(f^j(q),f^j(z))&\leq& d(f^j(q),f^{\tau p+j}(z))+d(f^{\tau p+j}(z),f^j(z))\\
&<& \delta/4+\delta/4=\delta/2.
\end{eqnarray*}
Therefore $d(f^{i}(q), f^i(x_r)) < \delta$ for all $a_r\leq  i \leq b_r$ and $1 \leq r \leq m$,

Choose an integer $T>0$ such that for each $\xi \in F$ we have
$$
\left|\frac{1}{pT}\sum_{i=0}^{pT-1}\xi (f^i(q))-\int_{Q(f)} \xi d\nu\right|<\frac{\eps}{8}.
$$
By the choice of $p$, we have $d(f^{np+i}(q),f^{i}(q))\leq \delta$ for all $i,n\geq 0$,
and hence
\begin{equation}
\left|\frac{1}{p}\sum_{i=0}^{p-1}\xi (f^i(q))-\int_{Q(f)} \xi d\nu\right|\leq\frac{\eps}{8}+
\left|\frac{1}{p}\sum_{i=0}^{p-1}\xi (f^i(q))-\frac{1}{pT}\sum_{i=0}^{pT-1}\xi (f^i(q))\right|< \frac{\eps}{4}.
\end{equation}
By our construction we have
\begin{eqnarray}
\left|\frac{1}{p}\sum_{r=1}^{m}\sum_{i=a_r}^{b_r}(\xi(f^i(q))-\xi(f^i(x_r)))\right|&\leq&\frac{mN\eps}{8p}
\leq \frac{\eps}{8},
\end{eqnarray}
and finally
\begin{equation} \label{eg:almostCalc-final}
\left|\frac{1}{p}\sum_{r=1}^{m}\sum_{i=b_r+1}^{b_r+M}\xi(f^i(q))\right|\leq \frac{Mm}{p}\leq \frac{\eps}{16}.
\end{equation}
Now, combining the above calculations \eqref{eg:almostCalc}--\eqref{eg:almostCalc-final}, we obtain that
\begin{eqnarray*}
\left|\int_X \xi d\mu - \int_X \xi d\nu\right|< \frac{\eps}{2}+\frac{\eps}{16}+\frac{\eps}{8}+\frac{\eps}{16}+\frac{\eps}{4}\leq \eps.
\end{eqnarray*}
This shows that $\nu \in W$ completing the proof.
\end{proof}

We finish our considerations with a few simple examples.
Note that if a dynamical system is positively expansive,
then Theorem~\ref{thm:shadowing-ETf} is a consequence of Theorem~\ref{thm:sieg}
by Corollary~\ref{cor:expansitive-psp}.
The same is also true for other types of expansiveness,
e.g. expansive homeomorphism or $c$-expansive surjections
as defined in~\cite{AH94}. We leave details to the reader.
A large class of systems with the shadowing property which are not
positively expansive can be found within the class of dynamical systems on the unit interval,
e.g. in the family of tent maps (e.g. see \cite{Coven})
In these maps, however, another argument can be used to prove
the periodic specification property (see Buzzi's proof of Blokh's classical result in~\cite{Buzzi}),
so ergodic measures in Theorem~\ref{thm:shadowing-ETf} are in fact supported on these points,
since Theorem~\ref{thm:sieg} works.

To obtain a system which satisfies assumption of Theorem~\ref{thm:shadowing-ETf}
but not of Theorem~\ref{thm:sieg} we need a little more work. To do so, we can perform the following standard construction.

\begin{exmp}\label{exmp:shadowing-non-period}
Let $(X_n,f_n)$ be a strongly mixing system with the shadowing property but without points of period $n$.
For example it can be a subshift of finite type defined by a graph with two distinct cycles of length $n+1$ and $n+2$
starting from the same vertex. Assume that the metric on $(X_n,d_n)$ satisfies $\diam (X_n)\leq 1$.
If we take an infinite Cartesian product
$(X,F)=(\prod_{n=1}^\infty X_n,\prod_{n=1}^\infty f_n)$ with the standard product metric $d(x,y)=\sum_{n=1}^\infty 2^{-n} d_n(x_n,y_n)$ then
it generates a topology compatible with the Thikhonov topology, in particular it is compact.
The map $(X,F)$ is strongly mixing and has the shadowing property as a product of maps with the same properties.
But it cannot have periodic points, because for any $x\in X$ and $n\in\mathbb{N}$
we have $F^n(x)_n=f^n(x_n)\neq x_n$.
\end{exmp}

In the above example the set consisting of regularly recurrent points whose orbits closures form subsystems conjugate to odometers is dense.

\begin{que}\label{que-1}
Can we ensure in Theorem~\ref{thm:spec} that $z\in RR(f)$ is such that $\overline{Orb(z,f)}$ is an odometer (up to conjugacy),
not only its almost 1-1 extension?
\end{que}

Note that if we can prove in Theorem~\ref{thm:spec} that $z$ is not only regularly recurrent
but also equicontinuous (in its orbit closure),
then immediately Question~\ref{que-1} has a positive answer.

\subsection*{Acknowledgement}
The first author was
partially supported by STU Scientific Research Foundation for
Talents (NTF12021), Guangdong Natural Science Foundation (S2013040014084) and  NNSF of China (11171320).
The research of second author was supported by Narodowe Centrum Nauki (National Science Center) in Poland, grant no. DEC-2011/03/B/ST1/00790.

The authors would like to thank Guohua Zhang for his helpful remarks.
The authors express many thanks to the anonymous referee, whose
remarks resulted in substantial improvements of the paper.

\end{document}